\documentclass{article}%
\usepackage{utopia}
\usepackage{amssymb,latexsym}
\usepackage{amsfonts}
\usepackage{amsmath}
\usepackage[colorlinks=true, pdfstartview=FitV, linkcolor=blue, citecolor=red,
urlcolor=blue]{hyperref}
\usepackage{color}
\usepackage{amssymb}
\usepackage{graphicx}%
\providecommand{\U}[1]{\protect\rule{.1in}{.1in}}
\newtheorem{theorem}{Theorem}

\newtheorem{corollary}[theorem]{Corollary}

\newtheorem{lemma}[theorem]{Lemma}

\newtheorem{proposition}[theorem]{Proposition}
\newtheorem{remark}{Remark}

\newenvironment{proof}[1][Proof]{\noindent\textbf{#1.} }{\ \rule{0.5em}{0.5em}}
\allowdisplaybreaks
\begin{document}

\title{Generalized harmonic numbers via poly-Bernoulli polynomials}
\author{Levent Karg\i n\thanks{lkargin@akdeniz.edu.tr}, Mehmet Cenkci \thanks{cenkci@akdeniz.edu.tr}
Ayhan Dil\thanks{adil@akdeniz.edu.tr} and M\"{u}m\"{u}n Can\thanks{mcan@akdeniz.edu.tr}
\\ {\small Department of Mathematics, Akdeniz University, Antalya, Turkey}}
\date{}
\maketitle

\begin{abstract}
We present a relationship between the generalized hyperharmonic numbers and the poly-Bernoulli
polynomials, motivated from the connections between harmonic and Bernoulli numbers. This relationship yields numerous identities for the hyper-sums
and several congruences.



\end{abstract}

\section{Introduction}
The $n$th harmonic number is defined by $H_{n}=1+\frac{1}{2}+\frac{1}%
{3}+\cdots+\frac{1}{n},$ where $H_{0}$ is conventionally understood to be
zero. The harmonic numbers naturally find places in mathematics and
applications such as combinatorics, mathematical analysis, number theory,
computer sciences. Therefore, introducing new representations and closed forms
for the harmonic numbers and their generalizations, relating harmonic numbers with other subjects are active research areas (see, for example \cite{B2,CaDa,C-M,DiMeCe,KC, Keller,W,Y1,Z}).

Among many other generalizations of the harmonic numbers, we are inte\-rested in
a unified generalization, the generalized hyperharmonic numbers, defined as
\[
H_{n}^{\left(  p,r\right)  }=\sum_{k=1}^{n}H_{k}^{\left(  p,r-1\right)  },
\]
with $H_{k}^{\left(  p,0\right)  }=1/k^{p}$ (see \cite{DiMeCe}). These numbers
extend two famous generalizations of the harmonic numbers, namely the
generalized harmonic numbers $H_{n}^{\left(  p,1\right)  }=H_{n}^{\left(
p\right)  }=\sum_{k=1}^{n}$ $1/k^{p}$ and the hyperharmonic numbers
$H_{n}^{\left(  1,r\right)  }=h_{n}^{(r)}=\sum_{k=1}^{n}h_{k}^{\left(
r-1\right)  }$.

There is an extensive literature on the harmonic and hyperharmonic numbers.
Among which we emphasize the formulas
\begin{equation}
\sum_{k=0}^{n}\left(  -1\right)  ^{k}
\genfrac{[}{]}{0pt}{}{n+1}{k+1}
B_{k}=n!H_{n+1}\label{1.1}%
\end{equation}
(see \cite{C-M} and see also \cite{CaDa,KC,W}) and
\begin{equation}
\sum_{k=0}^{n}
\genfrac{[}{]}{0pt}{}{n+r}{k+r}
_{r}B_{k}=n!h_{n+1}^{\left(  r-1\right)  }\label{HSB}%
\end{equation}
(see \cite{CaDa,KC}). Here $B_{n}$ is $n$th Bernoulli number defined by means
of the generating function%
\[
\frac{t}{e^{t}-1}=\sum_{n=0}^{\infty}B_{n}\frac{t^{n}}{n!},
\]
$
\genfrac{[}{]}{0pt}{}{n+r}{k+r}
_{r}$ is the $r$-Stirling number of the first kind defined by
\[
\left(  x+r\right)  ^{\overline{n}}=\left(  x+r\right)  \left(  x+r+1\right)  \cdots\left(  x+r+n-1\right)
=\sum\limits_{k=0}^{n}
\genfrac{[}{]}{0pt}{}{n+r}{k+r}
_{r}x^{k}%
\]
(see \cite{Broder1984}), and $
\genfrac{[}{]}{0pt}{}{n+1}{k+1}
_{1}=
\genfrac{[}{]}{0pt}{}{n+1}{k+1}
$ is the ordinary Stirling number of the first kind.

Equations (\ref{1.1}) and (\ref{HSB}) represent harmonic and hyperharmonic
numbers in terms of the Stirling and $r$-Stirling numbers of the first kind,
and Bernoulli numbers. These then give rise to the natural question of
representing generalized harmonic and generalized hyperharmonic numbers as
similar formulas. An affirmative answer to this question is given by the
following theorem, which represents the generalized hyperharmonic numbers in
terms of the $r$-Stirling numbers and poly-Bernoulli polynomials
$B_{n}^{\left(  p\right)  }\left(  x\right)  .$ The polynomials $B_{n}%
^{\left(  p\right)  }\left(  x\right)  $ are defined by
\begin{equation}
\sum_{n=0}^{\infty}B_{n}^{\left(  p\right)  }\left(  x\right)  \frac{t^{n}%
}{n!}=\frac{\mathrm{Li}_{p}\left(  1-e^{-t}\right)  }{1-e^{-t}}e^{xt}%
\label{gf-pB}%
\end{equation}
(see \cite{Bayad}), where $\mathrm{Li}_{p}\left(  t\right)  =\sum
_{n=1}^{\infty}t^{n}/n^{p}$ stands for the polylogarithm function.
\begin{theorem}
\label{teo3}For all non-negative integers $n$ and $r$, we have
\begin{equation}
\sum_{k=0}^{n}
\genfrac{[}{]}{0pt}{}{n+r}{k+r}
_{r}B_{k}^{\left(  p\right)  }\left(  q\right)  =n!H_{n+1}^{\left(
p,q+r\right)  }. \label{gh-pB}%
\end{equation}
\end{theorem}
As a result, we deduce the following identity which answers the question of
which type of the Bernoulli numbers are related to the generalized harmonic
numbers $H_{n}^{\left(  p\right)  }$.

\begin{corollary}
For any non-negative integer $n$%
\begin{equation}
\sum_{k=0}^{n}
\genfrac{[}{]}{0pt}{}{n+1}{k+1}
B_{k}^{\left(  p\right)  }=n!H_{n+1}^{\left(  p\right)  }, \label{hpb}%
\end{equation}
where $B_{k}^{\left(  p\right)  }=B_{k}^{\left(  p\right)  }\left(  0\right)
$ is the $k$th poly-Bernoulli number (see \cite{Kaneko}).
\end{corollary}

We also observe the following correspondence between the generalized
hyperharmonic numbers and the hyper-sums $S_{p}^{\left(  q\right)  }\left(
n\right)  $:%
\begin{equation}
H_{n}^{\left(  -p,q+1\right)  }=S_{p}^{\left(  q\right)  }\left(  n\right)
.\label{HSq}%
\end{equation}
The hyper-sum is introduced by Faulhaber as%
\[
S_{p}^{\left(  r\right)  }\left(  n\right)  =\sum_{k=1}^{n}S_{p}%
^{(r-1)}\left(  k\right)  ,
\]
with $S_{p}^{\left(  0\right)  }\left(  n\right)  =S_{p}\left(  n\right)
=1^{p}+2^{p}+\cdots+n^{p}$ (see \cite{Knuth}). The sums of powers of integers $S_{p}\left(
n\right)  $ have been interested since the classical times, for details see
\cite{GKP,G,Knuth,Merca}. Some recent studies on the hyper-sums include
explicit formulas, connection with the Ber\-noulli numbers, congruences,
generating functions, and recurrence formulas (\cite{Cereceda2,
Cereceda,Inibia,Laissaoui-Bounebirat-Rahmani2017}). The surprising
correspondence (\ref{HSq}) gives rise to numerous identities such as%
\[
S_{p}^{\left(  q\right)  }\left(  n\right)  =\sum_{j=1}^{p}\left(  -1\right)
^{p+j}n
\genfrac{\{}{\}}{0pt}{}{p}{j}
\binom{q+n+j}{q+1+j}j!,
\]
and congruences like
\[
S_{p}^{\left(  q\right)  }\left(  n\right)  \equiv\binom{n+q+1}{q+2}\pmod p,
\]
for a prime number $p$.

Apart from the identities (\ref{1.1}) and (\ref{HSB}), the harmonic numbers
are related to the Bernoulli numbers and polynomials via
\begin{equation}
\sum_{k=1}^{n}%
\left(  -1\right)  ^{k-1}
\genfrac{[}{]}{0pt}{}{n}{k}%
kB_{k-1}=\frac{2n!}{n+1}H_{n} \label{1.2}%
\end{equation}
(see \cite{Keller}) and%
\begin{align}
\sum_{k=1}^{n}
\genfrac{[}{]}{0pt}{}{n+m}{k+m}
_{m}kB_{k-1}\left(  r\right)  \mathbf{=}n!\binom{n+r+m-1}{r+m-2}  &  \left\{
\left(  H_{n+r+m-1}-H_{r+m-2}\right)  ^{2}\right. \nonumber\\
&  \left.  -H_{n+r+m-1}^{\left(  2\right)  }+H_{r+m-2}^{\left(  2\right)
}\right\}  \label{6}%
\end{align}
(see \cite{KC}). Since $B_{k}^{\prime}\left(  x\right)  =kB_{k-1}\left(
x\right)  $, the terms $kB_{k-1}$ and $kB_{k-1}\left(  r\right)  $ in the
summands suggest whether
\[
\sum_{k=1}^{n}
\genfrac{[}{]}{0pt}{}{n+m}{k+m}
_{m}\left.  \frac{d^{l}}{dx^{l}}B_{k}\left(  x\right)  \right\vert _{x=r}%
\]
can be written in terms of the harmonic numbers of any kind. We deal with this
problem in Theorem \ref{teoder}. It is worthwhile to mention that in the proof
of Theorem \ref{teoder}, the question raised in \cite{Dil,Mez} on the general
form of higher order derivatives of $h_{n}^{\left(  x\right)  }$ with respect
to $x$ is answered in a different sense.

The organization of the paper is as follows. Section 2 is a preliminary
section in which we give notation and basic definitions needed in the paper.
In Section 3 we prove Theorems \ref{teo3} and \ref{teoder}, and obtain some
recurrence formulas for the hyperharmonic and generalized harmonic numbers. In
Section 4 utilizing (\ref{HSq}) we derive numerous formulas for the hyper-sums
$S_{p}^{\left(  q\right)  }\left(  n\right)  $. Moreover, new results for the
sums of powers of integers $S_{p}\left(  n\right)  $ are presented. We
conclude the paper with Section 5 where we present some congruences for the
generalized hyperharmonic numbers, hyperharmonic numbers and hyper-sums.

\section{Preliminaries}
The $r$-Stirling numbers, which are natural generalizations of the ordinary
Stirling numbers, may be defined either in a combinatorial or in an analytic
way. Analytic way includes the generating functions. The $r$-Stirling numbers
of the second kind, denoted by $
\genfrac{\{}{\}}{0pt}{}{n+r}{k+r}
_{r}$, are defined by means of the exponential generating function%
\begin{equation}
\sum_{n=k}^{\infty}
\genfrac{\{}{\}}{0pt}{}{n+r}{k+r}
_{r}\frac{z^{n}}{n!}=\frac{\left(  e^{z}-1\right)  ^{k}}{k!}e^{rz}
\label{egfrS}%
\end{equation}
(see \cite{Broder1984}). The $r$-Stirling numbers of the first and the second kind are related via the $r$-Stirling
transform:%
\begin{equation*}
b_{n}=\sum_{k=0}^{n}
\genfrac{[}{]}{0pt}{}{n+r}{k+r}
_{r}a_{k}\text{ if and only if }a_{n}=\sum_{k=0}^{n}\left(  -1\right)  ^{n-k}
\genfrac{\{}{\}}{0pt}{}{n+r}{k+r}
_{r}b_{k} %
\end{equation*}
(see \cite{Broder1984}). In particular $
\genfrac{\{}{\}}{0pt}{}{n}{k}
_{0}=
\genfrac{\{}{\}}{0pt}{}{n}{k}
,\text{ }
\genfrac{\{}{\}}{0pt}{}{n+1}{k+1}
_{1}=
\genfrac{\{}{\}}{0pt}{}{n+1}{k+1}
,
$
where $
\genfrac{\{}{\}}{0pt}{}{n}{k}
$ is the ordinary Stirling number of the second kind, and%
$
\genfrac{[}{]}{0pt}{}{n}{k}
_{0}=
\genfrac{[}{]}{0pt}{}{n}{k}
,\text{ }
\genfrac{[}{]}{0pt}{}{n+1}{k+1}
_{1}=
\genfrac{[}{]}{0pt}{}{n+1}{k+1},
$
where $
\genfrac{[}{]}{0pt}{}{n}{k}
$ is the ordinary Stirling number of the first kind.

We list some of the basic facts about the Stirling numbers in the following.
\begin{lemma}
\label{conlemma1}We have

\noindent\hspace{-0.09in}%
\begin{tabular}
[t]{lllll}%
(1) $%
\genfrac{[}{]}{0pt}{}{0}{0}%
=1$, & $\hspace{-0.1in}%
\genfrac{[}{]}{0pt}{}{n}{0}%
=%
\genfrac{[}{]}{0pt}{}{0}{n}%
=0$, $n>0$, & $\hspace{-0.1in}%
\genfrac{[}{]}{0pt}{}{n}{1}%
=\left(  n-1\right)  !$, & $\hspace{-0.1in}%
\genfrac{[}{]}{0pt}{}{n}{n-1}%
=\frac{n\left(  n-1\right)  }{2}$, & $\hspace{-0.1in}%
\genfrac{[}{]}{0pt}{}{n}{n}%
=1.\vspace{0.05in}$\\
(2) $%
\genfrac{\{}{\}}{0pt}{}{0}{0}%
=1$, & $\hspace{-0.1in}%
\genfrac{\{}{\}}{0pt}{}{n}{0}%
=%
\genfrac{\{}{\}}{0pt}{}{0}{n}%
=0$, $n>0$, & $\hspace{-0.1in}%
\genfrac{\{}{\}}{0pt}{}{n}{1}%
=%
\genfrac{\{}{\}}{0pt}{}{n}{n}%
=1$, & $\hspace{-0.1in}%
\genfrac{\{}{\}}{0pt}{}{n}{n-1}%
=\frac{n\left(  n-1\right)  }{2}$ & (\cite{C1974}).$\vspace{0.05in}$%
\end{tabular}

\noindent(3) $%
\genfrac{[}{]}{0pt}{}{n+r}{0+r}%
_{r}=r^{\overline{n}}$, $%
\genfrac{[}{]}{0pt}{}{n+r}{1+r}%
_{r}=n!h_{n}^{\left(  r\right)  }$, $%
\genfrac{[}{]}{0pt}{}{n+r}{n-1+r}%
_{r}=\frac{n\left(  n-1\right)  }{2}+nr$,\ $%
\genfrac{[}{]}{0pt}{}{n+r}{n+r}%
_{r}=1$, and $%
\genfrac{[}{]}{0pt}{}{n+n}{m+n}%
_{n}=\delta_{m,n}$, where $\delta_{m,n}$ stands for the Kronecker's delta
(\cite{Broder1984}).$\vspace{0.05in}$

\noindent(4) For any $r$, $%
\genfrac{\{}{\}}{0pt}{}{n+r}{k+r}%
_{r}\equiv%
\genfrac{[}{]}{0pt}{}{n+r}{k+r}%
_{r}\equiv0\pmod n$ for a prime number $n$, provided that $k=2,3,\ldots,n-1$
(\cite{Hsu-Shiue1998}).
\end{lemma}

The generating function of the generalized hyperharmonic numbers
$H_{n}^{\left(  p,q\right)  }$ is given by \textbf{\ }%
\begin{equation}
\sum_{n=0}^{\infty}H_{n}^{\left(  p,q\right)  }t^{n}=\frac{Li_{p}\left(
t\right)  }{\left(  1-t\right)  ^{q}} \label{Hgf}%
\end{equation}
(see \cite{DiMeCe}), which reduces to
\begin{equation}
\sum_{n=0}^{\infty}h_{n}^{\left(  q\right)  }t^{n}=-\frac{\ln\left(
1-t\right)  }{\left(  1-t\right)  ^{q}}\label{hhgf}
\end{equation}
(see \cite{CG}), the generating function of the hyperharmonic numbers
$h_{n}^{\left(  q\right)  }$, that is related to the harmonic numbers by%
\begin{equation}
h_{n}^{\left(  q\right)  }=\binom{n+q-1}{q-1}\left(  H_{n+q-1}-H_{q-1}\right)
. \label{4}%
\end{equation}

The poly-Bernoulli polynomials $B_{n}^{\left(  p\right)  }\left(  x\right)  $,
defined in (\ref{gf-pB}),\textbf{ }can be expressed in terms of the
poly-Bernoulli numbers $B_{n}^{\left(  p\right)}=B_{n}^{\left(  p\right)  }\left(  0\right)$ as%
\[
B_{n}^{\left(  p\right)  }\left(  x\right)  =\sum_{k=0}^{n}\binom{n}{k}%
B_{k}^{\left(  p\right)  }x^{n-k}.
\]
The poly-Bernoulli numbers can be also represented in terms of the Stirling
numbers of the second kind by%
\begin{equation}
B_{n}^{\left(  p\right)  }=\left(  -1\right)  ^{n}\sum_{k=0}^{n}%
\genfrac{\{}{\}}{0pt}{}{n}{k}%
\frac{\left(  -1\right)  ^{k}k!}{\left(  k+1\right)  ^{p}}\label{pBs}%
\end{equation}
(see \cite{Kaneko}). Hence%
\begin{equation}
B_{0}^{\left(  p\right)  }=1\text{, }B_{1}^{\left(  p\right)  }=\frac{1}%
{2^{p}}\text{, }B_{0}^{\left(  p\right)  }\left(  x\right)  =1\text{, }%
B_{1}^{\left(  p\right)  }\left(  x\right)  =x+\frac{1}{2^{p}},\label{3}%
\end{equation}
etc. The poly-Bernoulli polynomials and numbers, which are studied recently in
different directions (\cite{Bayad,CY,KoLM,Y2}), are generalizations of the
Bernoulli polynomials $B_{n}\left(  x\right)  $ and the Bernoulli numbers
$B_{n}$ in that $B_{n}^{\left(  1\right)  }\left(  x-1\right)  =B_{n}\left(
x\right)  $\textbf{\ }and $B_{n}^{\left(  1\right)  }=B_{n}$ with
$B_{1}^{\left(  1\right)  }=-B_{1}$.

\section{Generalized hyperharmonic numbers}
We start this section by proving Theorem \ref{teo3}.
\begin{proof}
[Proof of Theorem \ref{teo3}]Using (\ref{gf-pB}) and (\ref{Hgf}) we have
\begin{align}
\sum_{k=0}^{\infty}B_{k}^{\left(  p\right)  }\left(  q+1-r\right)  \frac
{t^{k}}{k!}  &  =\frac{\mathrm{Li}_{p}\left(  1-e^{-t}\right)  }{1-e^{-t}%
}e^{\left(  q+1-r\right)  t}\nonumber\\
&  =\sum_{n=0}^{\infty}\left(  -1\right)  ^{n}n!H_{n+1}^{\left(  p,q+1\right)
}\frac{\left(  e^{-t}-1\right)  ^{n}}{n!}e^{-rt}. \label{14}%
\end{align}
Here by considering (\ref{egfrS}) we obtain%
\[
\sum_{k=0}^{\infty}B_{k}^{\left(  p\right)  }\left(  q+1-r\right)  \frac
{t^{k}}{k!}=\sum_{k=0}^{\infty}\left(  \sum_{n=0}^{k}\left(  -1\right)  ^{k-n}%
\genfrac{\{}{\}}{0pt}{}{k+r}{n+r}%
_{r}n!H_{n+1}^{\left(  p,q+1\right)  }\right)  \frac{t^{k}}{k!}.
\]
Then, comparing the coefficients of $\frac{t^{k}}{k!}$ in the both sides of
the above equation gives%
\begin{equation}
B_{n}^{\left(  p\right)  }\left(  q+1-r\right)  =\sum_{k=0}^{n}\left(
-1\right)  ^{n-k}%
\genfrac{\{}{\}}{0pt}{}{n+r}{k+r}%
_{r}k!H_{k+1}^{\left(  p,q+1\right)  }. \label{pB-gh}%
\end{equation}
Applying the $r$-Stirling transform to the above equation, we obtain the
desired result (\ref{gh-pB}).
\end{proof}

It worths to note recent works \cite{OS1} and \cite{OS2}, which relate the
Stirling numbers of the first kind and poly-Bernoulli numbers in different manner.

The following interesting alternating sum%
\[
-H_{1}+H_{2}-H_{3}+\cdots-H_{2n-1}+H_{2n}=\frac{1}{2}H_{n}%
\]
motivates the next result, which follows from the relation%
\[
\mathrm{Li}_{p}\left(  -t\right)  +\mathrm{Li}_{p}\left(  t\right)
=2^{1-p}\mathrm{Li}_{p}\left(  t^{2}\right).
\]

\begin{proposition}
\label{prop1}We have%
\[
\sum_{k=0}^{2n}\left(  -1\right)  ^{k}\binom{q+k-1}{k}H_{2n-k}^{\left(
p,q\right)  }=\frac{1}{2^{p}}H_{n}^{\left(  p,q\right)  }.
\]
In particular,%
\[
\sum_{k=0}^{2n}\left(  -1\right)  ^{k}H_{2n-k}^{\left(  p\right)  }=\frac
{1}{2^{p}}H_{n}^{\left(  p\right)  }\text{ \ and \ }\sum_{k=0}^{2n}\left(
-1\right)  ^{k}\binom{q+k-1}{k}h_{2n-k}^{\left(  q\right)  }=\frac{1}{2}%
h_{n}^{\left(  q\right)  }.
\]
\end{proposition}

\begin{proposition}
We have%
\[
H_{n}^{\left(  p,q+1\right)  }=\sum_{k=0}^{n}\left(  -1\right)  ^{k}%
\binom{p-q}{k}H_{n-k}^{\left(  p,p+1\right)  }.
\]
In particular,%
\[
H_{n}^{\left(  p\right)  }=\sum_{k=0}^{n}\left(  -1\right)  ^{k}\binom{p}%
{k}H_{n-k}^{\left(  p,p+1\right)  }.
\]

\end{proposition}

\begin{proof}
We have%
\begin{align*}
\sum_{n=0}^{\infty}H_{n}^{\left(  p,q+1\right)  }t^{n}  &  =\frac
{\mathrm{Li}_{p}\left(  t\right)  }{\left(  1-t\right)  ^{p+1}}\left(
1-t\right)  ^{p-q}\\
&  =\sum_{n=0}^{\infty}\left(  \sum_{k=0}^{n}H_{k}^{\left(  p,p+1\right)
}\binom{p-q}{n-k}\left(  -1\right)  ^{n-k}\right)  t^{n},
\end{align*}
from which the desired result follows.
\end{proof}

We now turn our attention to the sum%
\[
\sum_{k=1}^{n}%
\genfrac{[}{]}{0pt}{}{n+r}{k+r}%
_{r}\left.  \frac{d^{l}}{dx^{l}}B_{k}\left(  x\right)  \right\vert _{x=r},
\]
which is a more general form of (\ref{1.1}), (\ref{1.2}) and (\ref{6}). To
evaluate this sum we first recall that
\begin{equation}
\sum_{k=0}^{\infty}\binom{m+k}{m}P\left(  r,m+k,m\right)  t^{k}=\frac{\left(
-\ln\left(  1-t\right)  \right)  ^{r}}{\left(  1-t\right)  ^{m+1}}
\label{Agf}%
\end{equation}
(see \cite{Z}), where
\[
P\left(  r,m+k,m\right)  =P_{r}\left(  H_{m+k}^{\left(  1\right)  }%
-H_{m}^{\left(  1\right)  },H_{m+k}^{\left(  2\right)  }-H_{m}^{\left(
2\right)  },\ldots,H_{m+k}^{\left(  r\right)  }-H_{m}^{\left(  r\right)
}\right)  ,
\]
and the polynomial $P_{n}\left(  x_{1},x_{2},\ldots,x_{n}\right)  $ is defined
by $P_{0}=1$ and
\[
P_{n}\left(  x_{1},x_{2},\ldots,x_{n}\right)  =\left(  -1\right)  ^{n}%
Y_{n}\left(  -0!x_{1},-1!x_{2},\ldots,-\left(  n-1\right)  !x_{n}\right)  ,
\]
where $Y_{n}$ is the exponential Bell polynomial \cite{C1974}. A first few of
them may be listed as
$ P_{1}\left(  x_{1}\right)  =x_{1},\text{ }P_{2}\left(  x_{1},x_{2}\right)
=x_{1}^{2}-x_{2},\text{ }P_{3}\left(  x_{1},x_{2},x_{3}\right)  =x_{1}%
^{3}-3x_{1}x_{2}+2x_{3}.$

\begin{theorem}
\label{teoder}For nonnegative integers $q,r,$ and $n$, we have%
\begin{align}
&  \sum_{k=l}^{n}%
\genfrac{[}{]}{0pt}{}{n+r}{k+r}%
_{r}k\left(  k-1\right)  \cdots\left(  k-l+1\right)  B_{k-l}\left(  q\right)
\nonumber\\
&  \ =n!\binom{n+q+r-1}{q+r-2}P\left(  l+1,n+q+r-1,q+r-2\right)  . \label{6a}%
\end{align}

\end{theorem}

\begin{proof}
Differentiating both sides of (\ref{hhgf}) with respect to $x$ $l$ times gives
\[
\sum_{n=0}^{\infty}\frac{d^{l}}{dx^{l}}h_{n}^{\left(  x+1\right)  }t^{n}%
=\frac{\left(  -\ln\left(  1-t\right)  \right)  ^{l+1}}{\left(  1-t\right)
^{x+1}}.
\]
Setting $x=q$ in the above equation and using (\ref{Agf}), we see that
\begin{equation}
\left.  \frac{d^{l}}{dx^{l}}h_{n}^{\left(  x+1\right)  }\right\vert
_{x=q}=\binom{q+n}{q}P\left(  l+1,q+n,q\right)  . \label{7}%
\end{equation}
For $p=1$ and $q\rightarrow x-1,$ with the use of $B_{k}^{\left(  1\right)
}\left(  x-1\right)  =B_{k}\left(  x\right)  ,$ (\ref{gh-pB}) turns into
\[
\frac{1}{n!}\sum_{k=0}^{n}%
\genfrac{[}{]}{0pt}{}{n+r}{k+r}%
_{r}B_{k}\left(  x\right)  =h_{n+1}^{\left(  x+r-1\right)  }.
\]
Iterating $B_{k}^{\prime}\left(  x\right)  =kB_{k-1}\left(  x\right)
$\textbf{\ }$l$ times on the left-hand side and observing (\ref{7}) for the
right-hand side\textbf{\ }yield the closed formula (\ref{6a}).
\end{proof}

It is seen that (\ref{6a}) reduces to (\ref{1.1}) for $r=q=1$ and $l=0$, and
to (\ref{6}) for $l=1$. Another demonstration of (\ref{6a}) is the following
example corresponding to $l=2$:%
\begin{align*}
&  \frac{1}{n!}\sum_{k=2}^{n}%
\genfrac{[}{]}{0pt}{}{n+r}{k+r}%
_{r}k\left(  k-1\right)  B_{k-2}\left(  q\right) \\
&  =\binom{n+q+r-1}{q+r-2}\left\{  \left(  H_{n+q+r-1}-H_{q+r-2}\right)
^{3}+2\left(  H_{n+q+r-1}^{\left(  3\right)  }-H_{q+r-2}^{\left(  3\right)
}\right)  \right. \\
&  \ \left.  \qquad\qquad\qquad\qquad\quad-3\left(  H_{n+q+r-1}-H_{q+r-2}%
\right)  \left(  H_{n+q+r-1}^{\left(  2\right)  }-H_{q+r-2}^{\left(  2\right)
}\right)  \right\}  ,
\end{align*}
in particular for $r=q=1$%
\[
\frac{1}{n!}\sum_{k=2}^{n}\left(  -1\right)  ^{k}%
\genfrac{[}{]}{0pt}{}{n+1}{k+1}%
k\left(  k-1\right)  B_{k-2}=\left(  H_{n+1}\right)  ^{3}-3H_{n+1}%
H_{n+1}^{\left(  2\right)  }+2H_{n+1}^{\left(  3\right)  }.
\]
\begin{remark}
We note that (\ref{7})\textbf{\ }gives an answer to the question raised in
\cite{Dil,Mez} on the general form of higher order derivatives of
$h_{n}^{\left(  x\right)  }$ with respect to $x$.
\end{remark}

\section{Hyper-sums of powers}
The hyper-sums of powers of integers $S_{p}^{\left(  q\right)  }\left(
n\right)  $ are defined recursively by%
\[
S_{p}^{\left(  q\right)  }\left(  n\right)  =\sum_{k=1}^{n}S_{p}%
^{(q-1)}\left(  k\right)  ,
\]
with the initial condition $S_{p}^{\left(  0\right)  }\left(  n\right)
=S_{p}\left(  n\right)  $. We note that $S_{p}^{\left(  q\right)  }\left(
n\right)  $ satisfies the relation
\[
S_{p}^{\left(  q\right)  }\left(  n\right)  =\sum_{k=1}^{n}\binom{n+q-k}%
{q}k^{p}
\]
(see \cite{Laissaoui-Bounebirat-Rahmani2017}). We also observe that
\[
H_{n}^{\left(  p,q+1\right)  }=\sum_{k=1}^{n}\binom{n+q-k}{q}\frac{1}{k^{p}}
\]
(see \cite[p. 1648]{DiMeCe}). These simply imply the relation%

\begin{equation}
H_{n}^{\left(  -p,q+1\right)  }=S_{p}^{\left(  q\right)  }\left(  n\right)  .
\label{HaHS}%
\end{equation}
Therefore, we can translate the results given for the generalized
hyperharmonic numbers $H_{n}^{\left(  p,q\right)  }$ to the hyper-sums
$S_{p}^{\left(  q\right)  }\left(  n\right)  $. For instance,
\begin{align*}
S_{p}^{\left(  q\right)  }\left(  n\right)   &  =\frac{1}{2^{p}}\sum
_{k=0}^{2n}\binom{q+2n-k}{2n-k}\left(  -1\right)  ^{k}S_{p}^{\left(  q\right)
}\left(  k\right)  ,\\
S_{p}^{\left(  q\right)  }\left(  n\right)   &  =\sum_{k=0}^{n}\binom
{n-k+p+q-1}{n-k}S_{p}^{\left(  p\right)  }\left(  k\right)  ,
\end{align*}
which follow from Proposition \ref{prop1}. In particular, we have%
\[
S_{p}\left(  n\right)  =\frac{1}{2^{p}}\sum_{k=0}^{2n}\left(  -1\right)
^{k}S_{p}\left(  k\right)  \text{\ and }S_{p}\left(  n\right)  =\sum_{k=0}%
^{n}\binom{n-k+p-1}{n-k}S_{p}^{\left(  p\right)  }\left(  k\right)  .
\]
Moreover, using the duality theorem $B_{n}^{\left(  -p\right)  }%
=B_{p}^{\left(  -n\right)  }$ for the poly-Bernoulli numbers \cite[Theorem
2]{Kaneko} in (\ref{hpb}), we obtain an evaluation formula for sums of powers
of integers:
\[
\frac{1}{n!}\sum_{k=0}^{n}%
\genfrac{[}{]}{0pt}{}{n+1}{k+1}%
B_{p}^{\left(  -k\right)  }=S_{p}\left(  n+1\right)  .
\]

Our next result utilizes the following lemma, which is a polynomial extension
of the Arakawa-Kaneko formula
\[
B_{n}^{\left(  -p\right)  }=\sum_{j=0}^{\min\left\{  n,p\right\}  }\left(
j!\right)  ^{2}
\genfrac{\{}{\}}{0pt}{}{p+1}{j+1}
\genfrac{\{}{\}}{0pt}{}{n+1}{j+1}%
\]
(see \cite[Theorem 2]{AK}).

\begin{lemma}
We have%
\begin{equation}
B_{n}^{\left(  -p\right)  }\left(  q\right)  =\sum_{j=0}^{\min\left\{  n,p\right\}  }\left(  j!\right)
^{2}
\genfrac{\{}{\}}{0pt}{}{p+1}{j+1}
\genfrac{\{}{\}}{0pt}{}{n+q+1}{j+q+1}
_{q+1}. \label{pB-s}%
\end{equation}

\end{lemma}

\begin{proof}
We can write (\ref{Hgf}) in the form
\[
\sum_{p=0}^{\infty}\sum_{k=1}^{\infty}\left(  -1\right)  ^{k}H_{k}^{\left(
-p,q+1\right)  }\left(  e^{-t}-1\right)  ^{k}\frac{y^{p}}{p!}=\left(
1-e^{-t}\right)  \frac{e^{\left(  q+2\right)  t+y}}{1-\left(  1-e^{t}\right)
\left(  1-e^{y}\right)  }%
\]
by setting $t\rightarrow1-e^{-t}$. On the other hand, we equivalently have
\[
  \sum_{p=0}^{\infty}\sum_{k=0}^{\infty}\left(  -1\right)  ^{k}%
H_{k+1}^{\left(  -p,q+1\right)  }\frac{\left(  e^{-t}-1\right)  ^{k}}{e^{rt}}%
\frac{y^{p}}{p!}
  =\sum_{j=0}^{\infty}\left(  j!\right)  ^{2}\frac{\left(  1-e^{t}\right)
^{j}}{j!}e^{\left(  q+2-r\right)  t}\frac{\left(  1-e^{y}\right)  ^{j}}%
{j!}e^{y}.
\]
Last equality, (\ref{14}) and (\ref{egfrS}) give the desired result.
\end{proof}

Now, taking $-p$ in (\ref{gh-pB}) and using (\ref{pB-s}) yield
\[
n!H_{n+1}^{\left(  -p,q+r\right)  }    =\sum_{j=0}^{n}\left(  j!\right)  ^{2}%
\genfrac{\{}{\}}{0pt}{}{p+1}{j+1}%
\sum_{k=j}^{n}%
\genfrac{[}{]}{0pt}{}{n+r}{k+r}%
_{r}%
\genfrac{\{}{\}}{0pt}{}{k+q+1}{j+q+1}%
_{q+1}.
\]
Utilizing the formula
\[
\sum_{k=j}^{n}%
\genfrac{[}{]}{0pt}{}{n+r}{k+r}%
_{r}%
\genfrac{\{}{\}}{0pt}{}{k+s}{j+s}%
_{s}=\frac{n!}{j!}\binom{n+r+s-1}{j+r+s-1}
\]
(see \cite[Theorems 3.7 and 3.11]{NyRa}), we obtain the following.

\begin{theorem}
\label{teo4}For non-negative integers $p$, $q$, and $n$, we have%
\begin{equation}
S_{p}^{\left(  q\right)  }\left(  n\right)  =\sum_{j=0}^{p}j!%
\genfrac{\{}{\}}{0pt}{}{p+1}{j+1}%
\binom{n+q}{j+q+1}. \label{8}%
\end{equation}

\end{theorem}

\begin{remark}
We note that a slightly different form of (\ref{8}) can be found in
\cite{Cereceda2, Cereceda}.
\end{remark}

Equation (\ref{8}) provides a natural extension of the expression
\[
S_{p}\left(  n\right)  =\sum_{j=0}^{p}
j!
\genfrac{\{}{\}}{0pt}{}{p}{j}
\binom{n+1}{j+1}
\]
(see \cite[p. 285]{Knuth}) to $S_{p}^{\left(  q\right)  }\left(  n\right)  $.

Next result is an alternative representation for $S_{p}^{\left(  q\right)
}\left(  n\right)  $, which extends
\begin{equation}
S_{p}\left(  n\right)  =\sum_{j=1}^{p}\left(  -1\right)  ^{p+j}%
\genfrac{\{}{\}}{0pt}{}{p}{j}%
\binom{n+j}{j+1}j! \label{10}%
\end{equation}
(see \cite[p. 285]{Knuth}).

\begin{theorem}
We have%
\begin{equation}
S_{p}^{\left(  q\right)  }\left(  n\right)  =\sum_{j=1}^{p}\left(  -1\right)
^{p+j}
\genfrac{\{}{\}}{0pt}{}{p}{j}
\binom{n+q+j}{q+j+1}j!. \label{11}%
\end{equation}

\end{theorem}

\begin{proof}
We set $-p$ for $p$ in (\ref{Hgf}). Then
\[
\sum_{n=0}^{\infty}H_{n}^{\left(  -p,q\right)  }t^{n}=\frac{1}{\left(
1-t\right)  ^{q}}\sum_{k=1}^{\infty}k^{p}t^{k}=\frac{1}{\left(  1-t\right)
^{q+1}}w_{p}\left(  \frac{t}{1-t}\right)  ,
\]
where $w_{n}\left(  x\right)  $ is the $n$th geometric polynomial, defined by
\begin{equation}
\sum_{k=1}^{\infty}k^{n}t^{k}=\frac{1}{1-x}w_{n}\left(  \frac{x}{1-x}\right)
\label{gp}%
\end{equation}
(see \cite{B}). Using (\ref{gp}), the relation $\left(  1+x\right)
w_{n}\left(  x\right)  =x\left(  -1\right)  ^{n}w_{n}\left(  -x-1\right)  ,$
$n>0$, (see \cite[Eq. (22)]{Kargin1}) and
\[
w_{p}\left(  x\right)  =\sum_{j=0}^{p}%
\genfrac{\{}{\}}{0pt}{}{p}{j}%
j!x^{j}
\]
(see \cite{B}), we reach at%
\[
\sum_{n=0}^{\infty}H_{n}^{\left(  -p,q\right)  }t^{n}=\sum_{n=1}^{\infty}%
t^{n}\sum_{j=0}^{p}%
\genfrac{\{}{\}}{0pt}{}{p}{j}%
\binom{n+k+q-1}{n-1}\left(  -1\right)  ^{j+p}j!,
\]
which implies the result.
\end{proof}

\begin{remark}
A slightly different form of (\ref{11}) may be written as%
\[
S_{p}^{\left(  q\right)  }\left(  n\right)  =\sum_{j=0}^{p}\left(  -1\right)
^{p+j}%
\genfrac{\{}{\}}{0pt}{}{p+1}{j+1}%
\binom{n+q+j+1}{q+j+1}j!,
\]
which can be obtained by utilizing
\[
\mathrm{Li}_{-p}\left(  t\right)  =\left(  -1\right)  ^{p+1}\sum_{k=0}^{p}k!%
\genfrac{\{}{\}}{0pt}{}{p+1}{k+1}%
\left(  \frac{-1}{1-t}\right)  ^{k+1}%
\]
in the generating function (\ref{Hgf}).
\end{remark}

Taking into account (\ref{HSq}), it can be seen from (\ref{pB-gh}) and
(\ref{11}) that\textbf{\ }%
\[
B_{n}^{\left(  -p\right)  }\left(  q\right)  =\sum_{j=1}^{p}%
\genfrac{\{}{\}}{0pt}{}{p}{j}%
j!\left(  -1\right)  ^{p+j}\sum_{k=0}^{n}\left(  -1\right)  ^{n-k}%
\genfrac{\{}{\}}{0pt}{}{n+1}{k+1}%
\left(  j+q+2\right)  ^{\overline{k}}.
\]
This reduces to
\begin{equation}
B_{n}^{\left(  -p\right)  }\left(  q\right)  =\sum_{j=1}^{p}%
\genfrac{\{}{\}}{0pt}{}{p}{j}%
\left(  -1\right)  ^{p+j}j!\left(  j+q+1\right)  ^{n} \label{12}%
\end{equation}
utilizing the formula
\[
\left(  x-r\right)  ^{n}=\sum_{n=k}\left(  -1\right)  ^{n-k}%
\genfrac{\{}{\}}{0pt}{}{n+r}{k+r}%
_{r}x^{\overline{k}}%
\]
(see \cite{Broder1984}). (\ref{12}) stands for a polynomial extension of the
Arakawa-Kaneko formula
\[
B_{n}^{\left(  -p\right)  }=\sum_{j=1}^{p}%
\genfrac{\{}{\}}{0pt}{}{p}{j}%
\left(  -1\right)  ^{p+j}j!\left(  j+1\right)  ^{n}%
\]
(cf. \cite{AK}).

We conclude this section by stating further results for hyper-sums and sums of
powers, and some of their consequences.

\begin{theorem}
We have%
\begin{align}
S_{p}^{\left(  q\right)  }\left(  n\right)   &  =\frac{1}{q!}\sum_{k=0}%
^{q}\left(  -1\right)  ^{k}%
\genfrac{[}{]}{0pt}{}{q+n+1}{k+n+1}%
_{n+1}S_{p+k}\left(  n\right)  ,\label{16}\\
S_{q}\left(  n\right)   &  =\sum_{k=0}^{q}\left(  -1\right)  ^{k}%
\genfrac{\{}{\}}{0pt}{}{q+n+1}{k+n+1}%
_{n+1}\binom{k+n}{k+1}k!,\label{17}\\
S_{q-1}\left(  n\right)   &  =\sum_{k=0}^{q}\left(  -1\right)  ^{k}%
\genfrac{\{}{\}}{0pt}{}{q+n+1}{k+n+1}%
_{n+1}k!h_{n}^{\left(  k+1\right)  }. \label{18}%
\end{align}

\end{theorem}

\begin{proof}
For (\ref{16}), we appeal to the generating function of $S_{p}^{\left(
k\right)  }\left(  n\right)  $ depending on the index $k$
\[
\sum_{k=0}^{\infty}S_{p}^{\left(  k\right)  }\left(  n\right)  z^{k}=\frac
{1}{\left(  1-z\right)  ^{n+1}}\sum_{j=1}^{n}j^{p}\left(  1-z\right)  ^{j}
\]
(cf. \cite{Laissaoui-Bounebirat-Rahmani2017}). We set $z\rightarrow1-e^{-t}$
and see that
\[
\sum_{k=0}^{\infty}\left(  -1\right)  ^{k}k!S_{p}^{\left(  k\right)  }\left(
n\right)  \frac{\left(  e^{-t}-1\right)  ^{k}}{k!}e^{-\left(  n+1\right)
t}=\sum_{j=1}^{n}j^{p}e^{-jt}.
\]
Hence, the proof follows from (\ref{egfrS}) and $r$-Stirling transform.

(\ref{18}) follows similarly by using the generating function
\[
\sum_{k=0}^{\infty}h_{n}^{\left(  k+1\right)  }z^{k}=\frac{1}{\left(
1-z\right)  ^{n+1}}\sum_{j=1}^{n}\frac{\left(  1-z\right)  ^{j}}{j}
\]
(see \cite{KC}).

Utilizing the identity
\[%
\genfrac{\{}{\}}{0pt}{}{q+n}{k+n}%
_{n}=\sum_{j=k}^{q}\binom{q}{j}%
\genfrac{\{}{\}}{0pt}{}{j}{k}%
n^{q-j}%
\]
(cf. \cite{Broder1984}) and (\ref{10}) we find that
\begin{align*}
&  \sum_{k=0}^{q}\left(  -1\right)  ^{k}k!%
\genfrac{\{}{\}}{0pt}{}{q+n+1}{k+n+1}%
_{n+1}\binom{k+n}{k+1}\\
&  =\sum_{j=0}^{q}\left(  -1\right)  ^{j}\binom{q}{j}\left(  n+1\right)
^{q-j}S_{j}\left(  q\right) \\
&  =\left(  n+1\right)  ^{q}\sum_{m=0}^{q}\left(  \sum_{j=0}^{q}\left(
-1\right)  ^{j}\binom{q}{j}m^{j}\left(  n+1\right)  ^{-j}\right)
=S_{q}\left(  n\right)  ,
\end{align*}
which is (\ref{17}).
\end{proof}

Equations (\ref{8}) and (\ref{16}) yield the following interesting identity.

\begin{proposition}%
\[
\sum_{j=1}^{q}\frac{\left(  -1\right)  ^{j}}{j}\binom{n}{j}\binom{n+q-j}%
{n}=\binom{n+q}{q}\left(  H_{n+q}-H_{q}-H_{n}\right)  .
\]

\end{proposition}

\begin{proof}
Taking $p=-1$ in (\ref{16}) gives
\[
q!h_{n}^{\left(  q+1\right)  }=\sum_{k=1}^{q}\left(  -1\right)  ^{k}%
\genfrac{[}{]}{0pt}{}{q+n+1}{k+n+1}%
_{n+1}S_{k-1}\left(  n\right)  +q!\binom{n+q}{q}H_{n},
\]
by Lemma \ref{conlemma1} (3). We now make use of (\ref{8}) with $q=0$ and
deduce that%
\begin{align*}
&  \sum_{k=1}^{q}\left(  -1\right)  ^{k}%
\genfrac{[}{]}{0pt}{}{q+n+1}{k+n+1}%
_{n+1}S_{k-1}\left(  n\right) \\
&  =\sum_{j=1}^{q}\left(  j-1\right)  !\binom{n}{j}\sum_{k=j}^{q}\left(
-1\right)  ^{k}%
\genfrac{[}{]}{0pt}{}{q+n+1}{k+n+1}%
_{n+1}%
\genfrac{\{}{\}}{0pt}{}{k}{j}%
.
\end{align*}
Here we appeal to the equation
\[
\sum_{k=j}^{q}\left(  -1\right)  ^{k-j}%
\genfrac{[}{]}{0pt}{}{q+r}{k+r}%
_{r}%
\genfrac{\{}{\}}{0pt}{}{k+s}{j+s}%
_{s}=\binom{q}{j}\left(  r-s\right)  ^{\overline{q-j}}%
\]
(see \cite{Broder1984,NyRa}). Hence,%
\begin{align*}
\sum_{k=1}^{q}\left(  -1\right)  ^{k}%
\genfrac{[}{]}{0pt}{}{q+n+1}{k+n+1}%
_{n+1}S_{k-1}\left(  n\right)   &  =\sum_{j=1}^{q}\left(  -1\right)
^{j}\left(  j-1\right)  !\binom{n}{j}\binom{q}{j}\left(  n+1\right)
^{\overline{q-j}}\\
&  =q!\sum_{j=1}^{q}\frac{\left(  -1\right)  ^{j}}{j}\binom{n}{j}\binom
{n+q-j}{n}.
\end{align*}
Thus, the proof follows from (\ref{4}).
\end{proof}
\section{Congruences}
In this section we present several congruences for the hyperharmonic numbers,
the generalized hyperharmonic numbers, and the hyper-sums. The motivation
rises from the formulas established in Sections 3 and 4, where hyperharmonic
numbers and hyper-sums are formulated in terms of the Stirling numbers.

Throughout by a congruence $x\equiv y\pmod m$, we mean that the rational
numbers $x=\frac{a}{b}$ and $y=\frac{c}{d}$ satisfy $x\equiv y\pmod m$ if and
only if $m|\left(  ad-bc\right)  $. We also call a rational number as an
$m$-integer whenever its denominator is not divisible by $m$.

We start by reviewing some basic facts about Bernoulli numbers and polynomials.

\begin{lemma}
\label{conlemma2}We have

\noindent(1) (von Staudt-Clausen) For $n=1$ and for any even integer $n\geq2$,
we have%
\[
B_{n}=A_{n}-\sum_{\left(  p-1\right)  |n}\frac{1}{p},
\]
where $A_{n}$ is an integer and $p$ is a prime number. Equivalently, we have%
\[
pB_{n}\equiv\left\{
\begin{array}
[c]{ll}%
-1\pmod p, & \text{if }p-1\text{ divides }n,\\
0\pmod p, & \text{if }p-1\text{ does not divide }n.
\end{array}
\right.
\]
(2) If $q$ is an integer, then $\frac{B_{n}\left(  q\right)  -B_{n}}{n}$ is an
integer (\cite[p. 6]{Rademacher1973}).
\end{lemma}

The following result is about a divisibility property for generalized
hyperharmonic numbers.

\begin{theorem}
\label{contheorem1}Let $n$ be an odd prime. For positive integers $p$ and $q$,
we have%
\[
n^{p}H_{n+1}^{\left(  p,q\right)  }\equiv q\pmod n\text{.}%
\]

\end{theorem}

\begin{proof}
Let $n$ be an odd prime and $p\geq2$. Setting $r=0$ in (\ref{gh-pB}), multiplying both sides by $n^{p-1}$ and separating out the terms with
$k=0,k=1,k=n-1$, and $k=n$, we have%
\begin{align*}
n^{p-1}n!H_{n+1}^{\left(  p,q\right)  }  &  =n^{p-1}%
\genfrac{[}{]}{0pt}{}{n}{0}%
B_{0}^{\left(  p\right)  }\left(  q\right)  +n^{p-1}%
\genfrac{[}{]}{0pt}{}{n}{1}%
B_{1}^{\left(  p\right)  }\left(  q\right) \\
&  +n^{p-1}%
\genfrac{[}{]}{0pt}{}{n}{n-1}%
B_{n-1}^{\left(  p\right)  }\left(  q\right)  +n^{p-1}%
\genfrac{[}{]}{0pt}{}{n}{n}%
B_{n}^{\left(  p\right)  }\left(  q\right)  +\sum_{k=2}^{n-2}%
\genfrac{[}{]}{0pt}{}{n}{k}%
B_{k}^{\left(  p\right)  }\left(  q\right)  .
\end{align*}
We consider each of these terms separately.

By Lemma \ref{conlemma1} (1), we have $n^{p-1}%
\genfrac{[}{]}{0pt}{}{n}{0}%
B_{0}^{\left(  p\right)  }\left(  q\right)  =0$.

By Lemma \ref{conlemma1} (1) and (\ref{3}), we have%
\[
n^{p-1}%
\genfrac{[}{]}{0pt}{}{n}{1}%
B_{1}^{\left(  p\right)  }\left(  q\right)  =n^{p-1}\left(  n-1\right)
!\left(  q+\frac{1}{2^{p}}\right)  .
\]

We write%
\[
B_{n-1}^{\left(  p\right)  }\left(  q\right)  =B_{n-1}^{\left(  p\right)
}+\sum_{m=0}^{n-2}\binom{n-1}{m}B_{m}^{\left(  p\right)  }q^{n-1-m}.
\]
Now, from (\ref{pBs}) we observe that $B_{m}^{\left(  p\right)  }$ is an
$n$-integer for $m=0,1,\ldots,n-2$. On the other hand, $n^{p}B_{n-1}^{\left(
p\right)  }$ is an $n$-integer satisfying $n^{p}B_{n-1}^{\left(  p\right)
}\equiv-1\pmod n.$ Thus, by Lemma \ref{conlemma1} (1), we conclude that%
\begin{align*}
n^{p-1}%
\genfrac{[}{]}{0pt}{}{n}{n-1}%
B_{n-1}^{\left(  p\right)  }\left(  q\right)   &  =n^{p}B_{n-1}^{\left(
p\right)  }\frac{n-1}{2}+n^{p}\frac{n-1}{2}\sum_{m=0}^{n-2}\binom{n-1}{m}%
B_{m}^{\left(  p\right)  }q^{n-1-m}\\
&  \equiv-\frac{n-1}{2}\pmod n.
\end{align*}

Since%
\[
B_{n}^{\left(  p\right)  }\left(  q\right)  =B_{n}^{\left(  p\right)
}+nB_{n-1}^{\left(  p\right)  }q+\sum_{m=0}^{n-2}\binom{n}{m}B_{m}^{\left(
p\right)  }q^{n-m},
\]
we have%
\[
n^{p-1}%
\genfrac{[}{]}{0pt}{}{n}{n}%
B_{n}^{\left(  p\right)  }\left(  q\right)  =n^{p-1}B_{n}^{\left(  p\right)
}+n^{p}B_{n-1}^{\left(  p\right)  }q+n^{p-1}\sum_{m=0}^{n-2}\binom{n}{m}%
B_{m}^{\left(  p\right)  }q^{n-m}.
\]
Since $p\geq2$ and $B_{m}^{\left(  p\right)  }$ is an $n$-integer for
$m=0,1,\ldots,n-2$, we obtain that%
\[
n^{p-1}\sum_{m=0}^{n-2}\binom{n}{m}B_{m}^{\left(  p\right)  }q^{n-m}%
\equiv0\pmod n.
\]
We also have $n^{p}B_{n-1}^{\left(  p\right)  }q\equiv-q\pmod n$. Now, in
\cite[Theorem 1]{AK} it has been also proved that $n^{p-1}B_{n}^{\left(  p\right)  }$ is
an $n$-integer with%
\[
n^{p-1}B_{n}^{\left(  p\right)  }\equiv\frac{1}{n}%
\genfrac{\{}{\}}{0pt}{}{n}{n-1}%
-\frac{n}{2^{p}}\pmod n.
\]
Since $%
\genfrac{\{}{\}}{0pt}{}{n}{n-1}%
=\frac{n\left(  n-1\right)  }{2}$, we then have $n^{p-1}B_{n}^{\left(
p\right)  }\equiv\frac{n-1}{2}\pmod n$. Thus, we arrive at%
\[
n^{p-1}%
\genfrac{[}{]}{0pt}{}{n}{n}%
B_{n}^{\left(  p\right)  }\left(  q\right)  \equiv\frac{n-1}{2}-q\pmod n.
\]

Finally we write%
\[
n^{p-1}\sum_{k=2}^{n-2}%
\genfrac{[}{]}{0pt}{}{n}{k}%
B_{k}^{\left(  p\right)  }\left(  q\right)  =n^{p-1}\sum_{k=2}^{n-2}%
\genfrac{[}{]}{0pt}{}{n}{k}%
\sum_{m=0}^{k}\binom{k}{m}B_{m}^{\left(  p\right)  }q^{k-m}.
\]
By Lemma \ref{conlemma1} (4) and the fact that $B_{m}^{\left(  p\right)  }$ is
an $n$-integer for $m=0,1,\ldots,k$, where $k=2,3,\ldots,n-2$, we find that%
\[
n^{p-1}\sum_{k=2}^{n-2}%
\genfrac{[}{]}{0pt}{}{n}{k}%
B_{k}^{\left(  p\right)  }\left(  q\right)  \equiv0\pmod n.
\]

Combining these results we conclude that%
\begin{align*}
n^{p-1}n!H_{n+1}^{\left(  p,q\right)  }  &  \equiv n^{p-1}\left(  n-1\right)
!\left(  q+\frac{1}{2^{p}}\right)  -\frac{n-1}{2}+\frac{n-1}{2}-q\\
&  \equiv-q\pmod n.
\end{align*}
The result follows from Wilson's theorem.

Now for an odd prime $n$, let $p=1$ and $r=0$ in (\ref{gh-pB}). Since
$B_{k}^{\left(  1\right)  }\left(  q\right)  =B_{k}\left(  q+1\right)  $, with
$B_{1}\left(  q+1\right)  =q+\frac{1}{2}$, we have
\begin{align*}
n!H_{n+1}^{\left(  1,q\right)  }
&  =%
\genfrac{[}{]}{0pt}{}{n}{0}%
B_{0}\left(  q+1\right)  +%
\genfrac{[}{]}{0pt}{}{n}{1}%
B_{1}\left(  q+1\right)  +%
\genfrac{[}{]}{0pt}{}{n}{n-1}%
B_{n-1}\left(  q+1\right) \\
&  \quad+%
\genfrac{[}{]}{0pt}{}{n}{n}%
B_{n}\left(  q+1\right)  +\sum_{k=2}^{n-2}%
\genfrac{[}{]}{0pt}{}{n}{k}%
B_{k}\left(  q+1\right) \\
&  =\left(  n-1\right)  !\left(  q+\frac{1}{2}\right)  +\frac{n\left(
n-1\right)  ^{2}}{2}\frac{B_{n-1}\left(  q+1\right)  -B_{n-1}}{n-1}+\frac
{n-1}{2}nB_{n-1}\\
&  \quad+n\frac{B_{n}\left(  q+1\right)  -B_{n}}{n}+B_{n}+\sum_{k=2}^{n-2}k%
\genfrac{[}{]}{0pt}{}{n}{k}%
\frac{B_{k}\left(  q+1\right)  -B_{k}}{k}+\sum_{k=2}^{n-2}%
\genfrac{[}{]}{0pt}{}{n}{k}%
B_{k}.
\end{align*}
By Lemma \ref{conlemma2} (2),%
\[
\frac{n\left(  n-1\right)  ^{2}}{2}\frac{B_{n-1}\left(  q+1\right)  -B_{n-1}%
}{n-1}\equiv n\frac{B_{n}\left(  q+1\right)  -B_{n}}{n}\equiv0\pmod n,
\]
and $B_{n}=0$ since $n$ is an odd prime. On the other hand, by Lemma
\ref{conlemma2} (1), $nB_{n-1}\equiv-1\pmod n,$ and $B_{k}$ is an $n$-integer
for $k=2,3,\ldots,n-2$. Then, by Lemma \ref{conlemma1} (4), we have%
\[
n!H_{n+1}^{\left(  1,q\right)  }\equiv\left(  n-1\right)  !\left(  q+\frac
{1}{2}\right)  -\frac{n-1}{2}\equiv\left(  n-1\right)  !\left(  q+\frac{1}%
{2}+\frac{n-1}{2}\right)  \pmod n
\]
by Wilson's theorem. Canceling $\left(  n-1\right)  !$ gives the desired result.
\end{proof}

Next we present a symmetric-type congruence for hyperharmonic numbers.

\begin{theorem}
\label{contheorem2}For a prime $q$ and an integer $n$ with $1\leq n\leq q-1$,
we have%
\[
q\left(  h_{n}^{\left(  q+1\right)  }+nh_{q}^{\left(  n+1\right)  }\right)
\equiv n\pmod q.
\]

\end{theorem}

\begin{proof}
Let $1\leq n\leq q-1$ and $q$ be a prime. Applying the $r$-Stirling transform in (\ref{18}) we find that%
\begin{align*}
q!h_{n}^{\left(  q+1\right)  } &  =\sum_{k=0}^{q}\left(  -1\right)  ^{k}%
\genfrac{[}{]}{0pt}{}{q+n+1}{k+n+1}%
_{n+1}S_{k-1}\left(  n\right)  \\
&  =%
\genfrac{[}{]}{0pt}{}{q+n+1}{0+n+1}%
_{n+1}H_{n}-%
\genfrac{[}{]}{0pt}{}{q+n+1}{1+n+1}%
_{n+1}S_{0}\left(  n\right)  +\left(  -1\right)  ^{q}S_{q-1}\left(  n\right)
\\
&  +\sum_{k=2}^{q-1}\left(  -1\right)  ^{k}%
\genfrac{[}{]}{0pt}{}{q+n+1}{k+n+1}%
_{n+1}S_{k-1}\left(  n\right)  \\
&  \equiv\left(  n+1\right)  \cdots\left(  n+q\right)  H_{n}-nq!h_{q}^{\left(
n+1\right)  }+\left(  -1\right)  ^{q}S_{q-1}\left(  n\right)  \pmod q
\end{align*}
by Lemma \ref{conlemma1} (3) and (4). Now, $1\leq n\leq q-1$ implies that%
\[
\left(  n+1\right)  \left(  n+2\right)  \cdots\left(  n+q\right)  H_{n}%
\equiv0\pmod q
\]
and%
\[
S_{q-1}\left(  n\right)  =1^{q-1}+2^{q-1}+\cdots+n^{q-1}\equiv n\pmod q.
\]
Thus,%
\[
q!\left(  h_{n}^{\left(  q+1\right)  }+nh_{q}^{\left(  n+1\right)  }\right)
\equiv\left(  -1\right)  ^{q}n\pmod q.
\]
If $q=2$, then $n=1$, and we have%
\[
2\left(  h_{1}^{\left(  3\right)  }+h_{2}^{\left(  2\right)  }\right)
=2\left(  1+\frac{5}{2}\right)  \equiv1\pmod 2.
\]
Otherwise, $q!\left(  h_{n}^{\left(  q+1\right)  }+nh_{q}^{\left(  n+1\right)
}\right)  \equiv-n\pmod q$ and the result follows from Wilson's theorem.
\end{proof}

We conclude this section by stating two congruences about $S_{p}^{\left(  q\right)  }\left(  n\right)$. In
\cite{Laissaoui-Bounebirat-Rahmani2017} congruences for $S_{p}^{\left(
q\right)  }\left(  n\right)  $ when $n$ is a prime number were given. Here we
give results modulo a prime number $p$.

\begin{theorem}
\label{contheorem3}For a prime number $p$, we have%
\[
S_{p}^{\left(  q\right)  }\left(  n\right)  \equiv\binom{n+q+1}{q+2}\pmod p.
\]

\end{theorem}

\begin{proof}
Let $p$ be a prime in (\ref{11}). Then%
\begin{align*}
S_{p}^{\left(  q\right)  }\left(  n\right)   &  =\left(  -1\right)  ^{p+1}%
\genfrac{\{}{\}}{0pt}{}{p}{1}%
\binom{n+q+1}{q+2}+%
\genfrac{\{}{\}}{0pt}{}{p}{p}%
\binom{n+q+p}{q+p+1}p!\\
&  +\sum_{j=2}^{p-1}\left(  -1\right)  ^{p+j}%
\genfrac{\{}{\}}{0pt}{}{p}{j}%
\binom{n+q+j}{q+j+1}j!.
\end{align*}
By Lemma \ref{conlemma1} (2) and (4), we obtain the desired result.
\end{proof}

Next result seems obvious, but we record it here as an application of the
formula for the hyper-sums in terms of the poly-Bernoulli polynomials with
negative index. For this we need the following result about the poly-Bernoulli polynomials, which follows from (\ref{12}), Lemma \ref{conlemma1} (2) and (4).

\begin{lemma}
\label{conlemma3}Given a prime number $p$, a positive integer $n$, and a
nonnegative integer $q$, we have%
\[
B_{n}^{\left(  -p\right)  }\left(  q\right)  \equiv\left(  q+2\right)
^{n}\pmod p.
\]
\end{lemma}

\begin{proposition}
\label{contheorem4}Given a prime number $p$ and a nonnegative integer $q$, we
have%
\[
pS_{p}^{\left(  q\right)  }\left(  p+1\right)  \equiv0\pmod p.
\]

\end{proposition}

\begin{proof}
With the use of (\ref{HaHS}) in Theorem \ref{teo3} for $r=0$, we find that%
\[
n!S_{p}^{\left(  q\right)  }\left(  n+1\right)  =\sum_{k=0}^{n}%
\genfrac{[}{]}{0pt}{}{n}{k}%
B_{k}^{\left(  -p\right)  }\left(  q\right)  .
\]
Now let $p$ be a prime and set $n=p$. By Lemma \ref{conlemma1} (1) and (4),%
\begin{align*}
p!S_{p}^{\left(  q\right)  }\left(  p+1\right)   &  =%
\genfrac{[}{]}{0pt}{}{p}{0}%
B_{0}^{\left(  -p\right)  }\left(  q\right)  +%
\genfrac{[}{]}{0pt}{}{p}{1}%
B_{1}^{\left(  -p\right)  }\left(  q\right)  +%
\genfrac{[}{]}{0pt}{}{p}{p}%
B_{p}^{\left(  -p\right)  }\left(  q\right)  +\sum_{k=2}^{p-1}%
\genfrac{[}{]}{0pt}{}{p}{k}%
B_{k}^{\left(  -p\right)  }\left(  q\right) \\
&  \equiv\left(  p-1\right)  !B_{1}^{\left(  -p\right)  }\left(  q\right)
+B_{p}^{\left(  -p\right)  }\left(  q\right)  \pmod p.
\end{align*}
By Lemma \ref{conlemma3} and Wilson's theorem, we conclude that%
\[
pS_{p}^{\left(  q\right)  }\left(  n+1\right)  \equiv-\left(  q+2\right)
\left[  \left(  q+2\right)  ^{p-1}-1\right]  \pmod p,
\]
and the result follows from Fermat's little theorem.
\end{proof}


\begin{thebibliography}{99}                                                                                               %


\bibitem {AK}Arakawa T, Kaneko M. On poly-Bernoulli numbers, Comment. Math.
Univ. St. Paul. 48 (2) (1999) 159--167.

\bibitem {Bayad}Bayad A, Hamahata Y. Polylogarithms and poly-Bernoulli
polynomials, Kyushu J. Math. 65 (2011) 15--24.

\bibitem {B}Boyadzhiev KN. A series transformation formula and related
polynomials, Int. J. Math. Math. Sci. 23 (2005) 3849--3866.

\bibitem {B2}Boyadzhiev KN. Notes on the Binomial Transform, World Scientific,
Singapore, 2018.

\bibitem {Broder1984}Broder AZ. The $r$-Stirling numbers, Discrete Math. 49
(1984) 241--259.

\bibitem {CaDa}Can M, Da\u{g}l\i\ MC. Extended Bernoulli and Stirling matrices
and related combinatorial identities, Linear Algebra Appl. 444 (2014) 114--131.

\bibitem {CY}Cenkci M, Young PT. Generalizations of poly-Bernoulli and
poly-Cauchy numbers, European J. Math. 1 (2015) 799--828.

\bibitem {Cereceda2}Cereceda JL. Generalized Akiyama-Tanigawa algorithm for
hypersums of powers of integers, J. Integer Seq. 16 (2013) Article 13.3.2.

\bibitem {Cereceda}Cereceda JL. Iterative Procedure for hypersums of powers of
integers, J. Integer Seq. 17 (2014) Article 14.5.3.

\bibitem {C-M}Cheon G-S, El-Mikkawy MEA. Generalized harmonic numbers with
Riordan arrays, J. Number Theory 128 (2) (2008) 413--425.

\bibitem {C1974}Comtet L. Advanced Combinatorics, Riedel, Dordrech, Boston, 1974.

\bibitem {CG}Conway JH, Guy RK. The Book of Numbers. New York, USA:
Springer-Verlag, 1996.

\bibitem {Dil}Dil A. On the hyperharmonic function, SDU J. Nat. Appl. Sci. 23
(2019) 187--193.

\bibitem {DiMeCe}Dil A, Mez\H{o} I, Cenkci M. Evaluation of Euler-like sums
via Hurwitz zeta values, Turk. J. Math. 41 (2017) 1640--1655.

\bibitem {GKP}Graham RL, Knuth DE, Patashnik O. Concrete Mathematics,
Addison-Wesley, Reading, MA, 1989.

\bibitem {G}Gould HW. Evaluation of sums of convolved powers using Stirling
and Eulerian numbers, Fibonacci Quart. 16 (1978) 488--497.

\bibitem {Hsu-Shiue1998}Hsu LC, Shiue PJ-S. A unified approach to generalized
Stirling numbers, Adv. Appl. Math. 20 (1998) 366--384.

\bibitem {Inibia}Inaba Y. Hyper-sums of powers of integers and the
Akiyama-Tanigawa matrix, J. Integer Seq. 8 (2005) Article 05.2.7.

\bibitem {Kaneko}Kaneko M. Poly-Bernoulli numbers, J. Th\'{e}or. Nombres
Bordeaux 9 (1997) 221--228.

\bibitem {Kargin1}Karg\i n L. Some formulae for products of geometric
polynomials with applications, J. Integer Seq. 20 (2017) Article 17.4.4.

\bibitem {KC}Karg\i n L, Can M. Harmonic number identities via polynomials
with $r$-Lah coefficients, Comptes Rendus Mathematique, 358 (5) (2020) 535--550.

\bibitem {Keller}Kellner BC. Identities between polynomials related to
Stirling and harmonic numbers, Integers 14, \#A54 (2014)

\bibitem {Knuth}Knuth DE. Johann Faulhaber and sums of powers, Math. Comp. 61
(203) (1993) 277--294.

\bibitem {KoLM}Komatsu T, Liptai K, Mez\H{o} I. Incomplete poly-Bernoulli
numbers associated with incomplete Stirling numbers, Publ. Math. Debr. 88
(2016) 357--368.

\bibitem {Laissaoui-Bounebirat-Rahmani2017}Laissaoui D, Bounebirat F, Rahmani
M. On the hyper-sums of powers of integers, Miskolc Math. Notes, 18 (2017) 307--314.

\bibitem {Merca}Merca M. An alternative to Faulhaber's formula, Amer. Math.
Monthly 122 (6) (2015) 599--601.

\bibitem {Mez}Mez\H{o} I. Analytic extension of hyperharmonic numbers, Online
J. Anal. Comb. 4 (2009).

\bibitem {NyRa}Nyul G, R\'{a}cz G. The $r$-Lah numbers, Discrete Math. 338
(10) (2015) 1660--1666.

\bibitem {OS1}Y. Ohno, Y. Sasaki, Recurrence formulas for poly-Bernoulli
polynomials, Adv. Stud. Pure Math., 84 (2020), 353-360.

\bibitem {OS2}Y. Ohno, Y. Sasaki, Recursion formulas for poly-Bernoulli
numbers and their applications, Int. J. Number Theory, https://doi.org/10.1142/S1793042121500081.

\bibitem {Rademacher1973}Rademacher H. Topics in Analytic Number Theory,
Springer Verlag, New York, 1973.

\bibitem {W}Wang W. Riordan arrays and harmonic number identities, Comput.
Math. Appl. 60 (5) (2010) 1494--1509.

\bibitem {Y1}Young PT. Symmetries of Bernoulli polynomial series and
Arakawa--Kaneko zeta functions, J. Number Theory 143 (2014) 142--161.

\bibitem {Y2}Young PT. Bernoulli and poly-Bernoulli polynomial convolutions
and identities of $p$-adic Arakawa--Kaneko zeta functions, Int. J. Number
Theory 12 (05) (2016) 1295--1309.

\bibitem {Z}Zave DA. A series expansion involving the harmonic numbers,
Inform. Process. Lett. 5 (3) (1976) 75--77.
\end{thebibliography}
\end{document}